\documentclass[12pt,english]{amsart}

\usepackage[T1]{fontenc}
\usepackage[latin9]{inputenc}
\usepackage{geometry}
\usepackage{color}
\usepackage{amstext}
\usepackage{amsthm}
\usepackage{amssymb}

\makeatletter
\numberwithin{equation}{section}
\numberwithin{figure}{section}
\theoremstyle{plain}
\newtheorem{thm}{\protect\theoremname}
  \theoremstyle{plain}
  \newtheorem{prop}[thm]{\protect\propositionname}
  \theoremstyle{remark}
  \newtheorem{rem}[thm]{\protect\remarkname}
  \theoremstyle{definition}
  \newtheorem{example}[thm]{\protect\examplename}
  \theoremstyle{plain}
  \newtheorem{lem}[thm]{\protect\lemmaname}
  \theoremstyle{definition}
  \newtheorem{defn}[thm]{\protect\definitionname}
  \theoremstyle{plain}
  \newtheorem{cor}[thm]{\protect\corollaryname}

\usepackage{tikz}

\date{\today}

\usepackage{babel}
\providecommand{\corollaryname}{Corollary}
  \providecommand{\definitionname}{Definition}
  \providecommand{\examplename}{Example}
  \providecommand{\lemmaname}{Lemma}
  \providecommand{\propositionname}{Proposition}
  \providecommand{\remarkname}{Remark}
\providecommand{\theoremname}{Theorem}

\usepackage{babel}
\providecommand{\corollaryname}{Corollary}
  \providecommand{\definitionname}{Definition}
  \providecommand{\examplename}{Example}
  \providecommand{\lemmaname}{Lemma}
  \providecommand{\propositionname}{Proposition}
  \providecommand{\remarkname}{Remark}
\providecommand{\theoremname}{Theorem}

\newcommand{\cd}[3]{
\begin{tikzpicture}

\draw (1,0) -- (#1,0);
\foreach \p/\t/\w in {#2}
 {
  \node at (\p,-0.5) {$\scriptstyle{\w}$};
  \node at (\p,0) {\t};
 }
\foreach \s/\e/\r/\h in {#3}
 {
  \draw[thick] (\s, 0) -- (\s,\h-\r);
  \draw[thick] (\s,\h-\r) arc [radius=\r, start angle=180, end angle=90];
  \draw[thick] (\s+\r, \h) -- (\e-\r,\h);
  \draw[thick] (\e,\h-\r) arc [radius=\r, start angle=0, end angle=90];
  \draw[thick] (\e, 0) -- (\e,\h-\r);
 }
\end{tikzpicture}}

\newcommand{\cir}{\raisebox{-6pt}{$\circ$}}

\newcommand{\x}{$\times$}
\newcommand{\ve}{\raisebox{7pt}{$\vee$}}
\newcommand{\we}{\raisebox{-14pt}{$\wedge$}}

\makeatother

\usepackage{babel}
  \providecommand{\corollaryname}{Corollary}
  \providecommand{\definitionname}{Definition}
  \providecommand{\examplename}{Example}
  \providecommand{\lemmaname}{Lemma}
  \providecommand{\propositionname}{Proposition}
  \providecommand{\remarkname}{Remark}
\providecommand{\theoremname}{Theorem}

\begin{document}

\title{A Weyl-type character formula for PDC modules of $\mathfrak{gl}(m|n)$}

\author{Michael Chmutov, Crystal Hoyt, Shifra Reif}

\thanks{The third author was partially supported by NSF RTG grant DMS 0943832.}
\begin{abstract}
In 1994, Kac and Wakimoto suggested a generalization of Bernstein
and Leites character formula for basic Lie superalgebras, and the
natural question was raised: to which simple highest weight modules
does it apply? In this paper, we prove a similar formula for a large
class of finite-dimensional simple modules over the Lie superalgebra
$\mathfrak{gl}(m|n)$, which we call piecewise disconnected modules,
or PDC. The class of PDC modules naturally includes totally connected
modules and totally disconnected modules, the two families for which
similar character formulas were proven by Su and Zhang as special
cases of their general formula. This paper is part of our program
for the pursuit of elegant character formulas for Lie superalgebras. 
\end{abstract}

\maketitle

\section{Introduction}

It is well known that the theory of character formulas for simple
finite-dimensional modules over a Lie superalgebra is a nontrivial
extension of the classical case. The problem originates from the existence
of the so called \emph{atypical} roots. In the absence of these roots,
Kac proved in 1977 that the Weyl character formula generalizes in
a straightforward fashion \cite{K2,K3}. In 1980, an elegant Weyl-type
character formula was proven by Bernstein and Leites \cite{BL} for
simple highest weight modules of atypicality 1 (see Section \ref{sub:Atypical-modules.}).
Let $L\left(\lambda\right)$ be a finite-dimensional simple module
of highest weight $\lambda$ and atypical root $\beta$, then

\[
e^{\rho}R\cdot\mbox{ch }L\left(\lambda\right)=\sum_{w\in W}(-1)^{l(w)}w\left(\frac{e^{\lambda+\rho}}{1+e^{-\beta}}\right).
\]

Great efforts were invested in generalizing this formula to all finite-dimensional
modules of $\mathfrak{gl}\left(m|n\right)$. It was shown in \cite{VHKT}
that such a formula does not hold for all modules but does hold for
various families of modules, such as the covariant and contravariant
modules. In \cite{KW}, Kac and Wakimoto stated a similar formula
for the case when all of the atypical roots are simple. This was proven
by the authors in \cite{CHR} for $\mathfrak{gl}\left(m|n\right)$-
modules, and for modules over other Lie superalgebras in \cite{CK,GK}.
In \cite{SZ}, Su and Zhang gave a closed character formula for all
finite-dimensional $\mathfrak{gl}\left(m|n\right)$-modules, based
on the work of Serganova \cite{S1,S2} and Brundan \cite{B}. However,
this formula is rather intricate and difficult to apply. For modules
of atypicality $r$, the Su-Zhang formula consists of a\textcolor{black}{n
alternating sum of up to $r!\cdot2^{r-1}$ terms, each of} which resembles
the Kac-Wakimoto formula. Therefore, it is still a major goal to find
classes of modules which satisfy a simpler Weyl type character formula.

In this paper, we present a class of $\mathfrak{gl}\left(m|n\right)$-modules
for which we prove a Weyl type character formula. Our formula consists
of only one Kac-Wakimoto term. We call these modules \emph{piecewise
disconnected}, or PDC. This class of modules naturally extends the
two classes previously known to admit the Kac-Wakimoto formula: the
totally connected modules and totally disconnected modules (see \cite[Corollaries 4.13 ,4.15]{SZ}).
The PDC modules are the modules whose highest weight splits into components,
each of which resembles a totally connected module while the relation
between these components resembles a totally disconnected module (see
Definition \ref{def-pdc}).

The class of totally connected modules (also known as Kostant modules)
includes the covariant and contravariant modules and was shown by
the authors in \cite{CHR} to be precisely the same class of modules
for which the Kac-Wakimoto formula was originally conjectured in \cite[Section 3]{KW},
\cite[Conjecture 3.6]{KW2}. Totally connected modules and totally
disconnected modules were also studied over the queer Lie superalgebras
in \cite{CK2} where closed character formulas for these classes were
derived and proven.

Our main result is as follows. Let $L\left(\lambda\right)$ be a PDC
module of highest weight $\lambda$ with respect to the standard choice
of simple roots. We prove the following character formula for $L\left(\lambda\right)$:

\begin{equation}
e^{\rho}R\cdot\mbox{ch }L\left(\lambda\right)=\frac{(-1)^{|(\lambda^{\rho})^{\Uparrow}-\lambda^{\rho}|_{S_{\lambda}}}}{t_{\lambda}}\sum_{w\in W}(-1)^{l(w)}w\left(\frac{e^{\left(\lambda^{\rho}\right)^{\Uparrow}}}{\prod_{\beta\in S_{\lambda}}\left(1+e^{-\beta}\right)}\right),\label{eq:main formula}
\end{equation}
where $S_{\lambda}$ is a maximal orthogonal set of atypical roots;
the weight $\left(\lambda^{\rho}\right)^{\Uparrow}$ is obtained by
adding certain atypical roots to $\lambda+\rho$; the exponent $|(\lambda^{\rho})^{\Uparrow}-\lambda^{\rho}|_{S_{\lambda}}$
is the number of such roots added; and $t_{\lambda}$ is a positive
integer determined by the lengths of the atypical components $\lambda$
(see Definitions \ref{def-t_lambda}, \ref{def: uparrow} and \ref{def: number of betas}).

When the defect of $\mathfrak{gl}\left(m|n\right)$ is less than or
equal to $2$, (i.e. $m$ or $n$ is less than or equal to $2$),
then all modules are PDC and hence the above character formula applies.
The standard module of $\mathfrak{g}=\mathfrak{gl}\left(m|n\right)$
is totally connected, and hence PDC (see Example \ref{ex:tc standard}).
If $\mathfrak{g}=\mathfrak{gl}\left(m|n\right)$ has defect greater
than or equal to $3$, then the non-trivial simple subquotient of
the adjoint module of $\mathfrak{g}$ is not PDC (see Example \ref{ex:non-pdc adjoint}).
However, the module $\mathfrak{g}\otimes\mathfrak{g\otimes\cdots\otimes g}$
obtained by tensoring $n-1$ copies of the adjoint module contains
a simple subquotient that is PDC but is neither totally connected
nor totally disconnected (see Example \ref{ex:pdc not tc not tdc}).

Our proof of the above character formula uses Brundan's algorithm
for computing Kazhdan-Lusztig Polynomials \cite{B} and is based on
ideas from \cite{SZ}. Unlike the totally connected and totally disconnected
cases, for a general piecewise disconnected weight $\lambda$, the
weight $\left(\lambda^{\rho}\right)^{\Uparrow}$ appearing in formula
(\ref{eq:main formula}) does not correspond to a highest weight vector
for any choice of simple roots.

\section{Preliminaries}

\subsection{The general linear Lie superalgebra.}

Let $\mathfrak{g}$ denote the general linear Lie superalgebra $\mathfrak{gl}(m|n)$
over the complex field $\mathbb{C}$. As a vector space, $\mathfrak{g}$
can be identified with the endomorphism algebra $\mbox{End}(V_{\bar{0}}\oplus V_{\bar{1}})$
of a $\mathbb{Z}_{2}$-graded vector space $V_{\bar{0}}\oplus V_{\bar{1}}$
with $\mbox{dim }V_{\bar{0}}=m$ and $\mbox{dim }V_{\bar{1}}=n$.
Then $\mathfrak{g}=\mathfrak{g}_{\bar{0}}\oplus\mathfrak{g}_{\bar{1}}$,
where 
\[
\mathfrak{g}_{\bar{0}}=\mbox{End}(V_{\bar{0}})\oplus\mbox{End}(V_{\bar{1}})\quad\mathrm{and}\quad\mathfrak{g}_{\bar{1}}=\mbox{Hom}(V_{\bar{0}},V_{\bar{1}})\oplus\mbox{Hom}(V_{\bar{1}},V_{\bar{0}}).
\]
A homogeneous element $x\in\mathfrak{g}_{\bar{0}}$ has degree $0$,
denoted $\mbox{deg}(x)=0$, while $x\in\mathfrak{g}_{\bar{1}}$ has
degree $1$, denoted $\mbox{deg}(x)=1$. We define a bilinear operation
on $\mathfrak{g}$ by letting 
\[
[x,y]=xy-(-1)^{\mathrm{deg}(x)\mathrm{deg}(y)}yx
\]
on homogeneous elements and then extending linearly to all of $\mathfrak{g}$.

By fixing a basis of $V_{\bar{0}}$ and $V_{\bar{1}}$, we can realize
$\mathfrak{g}$ as the set of $(m+n)\times(m+n)$ matrices, where
\[
\mathfrak{g}_{\bar{0}}=\left\{ \left(\begin{array}{cc}
A & 0\\
0 & B
\end{array}\right)\mid A\in M_{m,m},\ B\in M_{n,n}\right\} \text{ and }\mathfrak{g}_{\bar{1}}=\left\{ \left(\begin{array}{cc}
0 & C\\
D & 0
\end{array}\right)\mid C\in M_{m,n},\ D\in M_{n,m}\right\} ,
\]
and $M_{r,s}$ denotes the set of $r\times s$ matrices.

\subsection{Root space decomposition and choice of simple roots}

The Cartan subalgebra $\mathfrak{h}$ of $\mathfrak{g}$ is the set
of diagonal matrices, and it has a natural basis 
\[
\{E_{1,1},\ldots,E_{m,m};E_{m+1,m+1},\ldots,E_{m+n,m+n}\},
\]
where $E_{ij}$ denotes the matrix whose $ij$-entry is $1$ and all
other entries are $0$. Fix the dual basis $\{\varepsilon_{1},\ldots,\varepsilon_{m};\delta_{1},\ldots,\delta_{n}\}$
for $\mathfrak{h}^{*}$. We define a bilinear form on $\mathfrak{h}^{*}$
by $(\varepsilon_{i},\varepsilon_{j})=\delta_{ij}=-(\delta_{i},\delta_{j})$
and $(\varepsilon_{i},\delta_{j})=0$.

Then $\mathfrak{g}$ has a root space decomposition $\mathfrak{g}=\mathfrak{h}\oplus\left(\bigoplus_{\alpha\in\Delta_{\bar{0}}}\mathfrak{g}_{\alpha}\right)\oplus\left(\bigoplus_{\alpha\in\Delta_{\bar{1}}}\mathfrak{g}_{\alpha}\right),$
where the set of roots of $\mathfrak{g}$ is $\Delta=\Delta_{\bar{0}}\cup\Delta_{\bar{1}}$,
with 
\[
\begin{aligned} & \Delta_{\bar{0}}=\{\varepsilon_{i}-\varepsilon_{j}\mid1\leq i\neq j\leq m\}\cup\{\delta_{k}-\delta_{l}\mid1\leq k\neq l\leq n\},\\
 & \Delta_{\bar{1}}=\{\pm(\varepsilon_{i}-\delta_{k})\mid1\leq i\leq m,\ 1\leq k\leq n\},
\end{aligned}
\]
and $\mathfrak{g}_{\varepsilon_{i}-\varepsilon_{j}}=\mathbb{C}E_{ij},\ \mathfrak{g}_{\delta_{k}-\delta_{l}}=\mathbb{C}E_{m+k,m+l},\ \mathfrak{g}_{\varepsilon_{i}-\delta_{k}}=\mathbb{C}E_{i,m+k},\ \mathfrak{g}_{\delta_{k}-\varepsilon_{i}}=\mathbb{C}E_{m+k,i}.$

The Weyl group of $\mathfrak{g}$ is $W=Sym(m)\times Sym(n)$, and
$W$ acts on $\mathfrak{h}^{*}$ by permuting the indices of the $\varepsilon$'s
and by permuting the indices of the $\delta$'s. In particular, the
even reflection $s_{\varepsilon_{i}-\varepsilon_{j}}$ interchanges
the $i$ and $j$ indices of the $\varepsilon$'s and fixes all other
indices, while $s_{\delta_{k}-\delta_{l}}$ interchanges the $k$
and $l$ indices of the $\delta$'s and fixes all other indices.

A set of simple roots $\pi\subset\Delta$ determines a decomposition
of $\Delta$ into positive and negative roots, $\Delta=\Delta^{+}\cup\Delta^{-}$.
There is a corresponding triangular decomposition of $\mathfrak{g}$
given by $\mathfrak{g}=\mathfrak{n}^{+}\oplus\mathfrak{h}\oplus\mathfrak{n}^{-}$,
where $\mathfrak{n}^{\pm}=\oplus_{\alpha\in\Delta^{\pm}}\mathfrak{g}_{\alpha}$.
Let $\Delta_{\bar{d}}^{+}=\Delta_{\bar{d}}\cap\Delta^{+}$ for $d\in\{0,1\}$.
For the rest of the paper, we fix the standard choice of simple roots
\[
\pi=\left\{ \varepsilon_{1}-\varepsilon_{2},\ldots,\varepsilon_{m-1}-\varepsilon_{m},\varepsilon_{m}-\delta_{1},\delta_{1}-\delta_{2}\ldots,\delta_{n-1}-\delta_{n}\right\} .
\]
The corresponding decomposition $\Delta=\Delta^{+}\cup\Delta^{-}$
is given by 
\begin{equation}
\Delta_{\bar{0}}^{+}=\{\varepsilon_{i}-\varepsilon_{j}\}_{1\leq i<j\leq m}\cup\{\delta_{k}-\delta_{l}\}_{1\leq k<l\leq n}\quad\text{ and}\quad\Delta_{\bar{1}}^{+}=\{\varepsilon_{i}-\delta_{k}\}_{1\leq i\leq m,\ 1\leq k\leq n}.\label{eq:set of positive roots}
\end{equation}
The standard choice of simple roots has the unique property that $W$
fixes $\Delta_{\bar{1}}^{+}$. Moreover, it contains a basis for $\Delta_{\bar{0}}^{+}$,
which we denote by $\pi_{\bar{0}}$.

Let $\rho=\frac{1}{2}\sum_{\alpha\in\Delta_{\bar{0}}^{+}}\alpha-\frac{1}{2}\sum_{\alpha\in\Delta_{\bar{1}}^{+}}\alpha.$
Then for $\alpha\in\pi$, we have $(\rho,\alpha)=(\alpha,\alpha)/2$.

We define the root lattice as $Q=\sum_{\alpha\in\pi}\mathbb{Z}\alpha$
and the positive root lattice as $Q^{+}=\sum_{\alpha\in\pi}\mathbb{N}\alpha$,
where $\mathbb{N}=\{0,1,2,\ldots\}$. A partial order is defined on
$\mathfrak{h}^{*}$ by $\mu>\nu$ when $\mu-\nu\in Q^{+}$.

\subsection{Finite dimensional modules for $\mathfrak{\mathfrak{\mathfrak{g}=gl(m|n)}}$}

For each weight $\lambda\in\mathfrak{h}^{*}$, the \emph{Verma module}
of highest weight $\lambda$ is the induced module $M(\lambda):=\mbox{Ind}_{\mathfrak{n}^{+}\oplus\mathfrak{h}}^{\mathfrak{g}}\mathbb{C}_{\lambda},$
\[
M(\lambda):=\mbox{Ind}_{\mathfrak{n}^{+}\oplus\mathfrak{h}}^{\mathfrak{g}}\mathbb{C}_{\lambda},
\]
where $\mathbb{C}_{\lambda}$ is the one-dimensional module such that
$h\in\mathfrak{h}$ acts by scalar multiplication of $\lambda(h)$
and $\mathfrak{n}^{+}$ acts trivially. The Verma module $M(\lambda)$
has a unique simple quotient, which we denote $L(\lambda)$.

For each $\lambda\in\mathfrak{h}^{*}$, let $L_{\bar{0}}(\lambda)$
denote the simple highest weight $\mathfrak{g}_{\bar{0}}$-module
with respect to $\pi_{\bar{0}}$. The {\em Kac module} of highest
weight $\lambda$ with respect to $\pi$ is the induced module 
\[
\overline{L}(\lambda):=\mbox{Ind}_{\mathfrak{g}_{\bar{0}}\oplus\mathfrak{n}_{\bar{1}}^{+}}^{\mathfrak{g}}L_{\bar{0}}(\lambda)
\]
defined by letting $\mathfrak{n}_{\bar{1}}^{+}:=\oplus_{\alpha\in\Delta_{\bar{1}}^{+}}\mathfrak{g}_{\alpha}$
act trivially on the $\mathfrak{g}_{\bar{0}}$-module $L_{\bar{0}}(\lambda)$.
The unique simple quotient of $\overline{L}(\lambda)$ is $L(\lambda)$.

Let $\mathfrak{h}_{\mathbb{R}}^{*}=\sum_{\alpha\in\pi}\mathbb{R}\alpha$
. A weight $\nu\in\mathfrak{h}_{\mathbb{R}}^{*}$ is called integral
(resp. dominant; strictly dominant) if $\langle\lambda,\alpha\rangle\in\mathbb{Z}$
(resp. $\langle\lambda,\alpha\rangle\geq0$; $\langle\lambda,\alpha\rangle>0$)
for all $\alpha\in\Delta_{\bar{0}}^{+}$, where $\langle\lambda,\alpha\rangle=\frac{2(\lambda,\alpha)}{(\alpha,\alpha)}$.

For a proof of the following proposition see for example \cite[14.1.1]{M}.
Given $\lambda\in\mathfrak{h}^{*}$, we use the following abbreviation
$\lambda^{\rho}:=\lambda+\rho$. 
\begin{prop}
\label{sub:rho^st strictly dominant}Let $\mathfrak{g}=\mathfrak{gl}(m|n)$
and $\lambda\in\mathfrak{h}^{*}$. Then, $L(\lambda)$ is a finite
dimensional $\mathfrak{g}$-module iff $L_{\bar{0}}(\lambda)$ is
finite dimensional $\mathfrak{g}_{\bar{0}}$-module iff the Kac module
$\overline{L}(\lambda)$ is finite dimensional iff $\lambda$ is a
dominant integral weight iff $\lambda^{\rho}$ is a strictly dominant
integral weight. 
\end{prop}
An element $\lambda\in\mathfrak{h_{\mathbb{R}}}^{*}$ is called \emph{regular}
if $(\nu,\varepsilon_{i})\neq(\nu,\varepsilon_{j})$ and $(\nu,\delta_{i})\neq(\nu,\delta_{j})$
for all $i\neq j$. An element $\nu\in\mathfrak{h}_{\mathbb{R}}^{*}$
is regular if and only if there exists $w\in W$ such that $w(\nu)$
is strictly dominant.

\subsection{Atypical modules.\label{sub:Atypical-modules.}}

Let $L(\lambda)$ be a finite dimensional $\mathfrak{g}$-module.
We call $\beta\in\Delta_{\overline{1}}$ \emph{atypical} if $\left(\lambda^{\rho},\beta\right)=\left(\beta,\beta\right)=0$.
The \textit{atypicality} of $L(\lambda)$ is the maximal number of
linearly independent roots $\beta_{1},...,\beta_{r}$ such that $\left(\beta_{i},\beta_{j}\right)=0$
and $\left(\lambda^{\rho},\beta_{i}\right)=0$ for $i,j=1,\ldots,r$.
Such a set $S_{\lambda}=\left\{ \beta_{1},\ldots,\beta_{r}\right\} $
is called a \textit{$\lambda^{\rho}$-maximal isotropic set}, and
we assume that the elements of $S_{\lambda}$ are ordered so that
$\beta_{i}=\varepsilon_{p_{i}}-\delta_{q_{i}}$ and $q_{i}<q_{i+1}$.
As in \cite{KW}, we denote the \textit{\emph{atypicality}} of $L(\lambda)$
by $\mathrm{atp}(\lambda^{\rho})=r$. The module $L(\lambda)$ is
called \textit{typical} if this set is empty, and \textit{atypical}
otherwise. For the standard choice of simple roots the set $S_{\lambda}$
is uniquely determined.

Let $P$ denote the set of integral weights, $P^{+}$ the set of dominant
integral weights, and define 
\[
\mathbb{P}^{+}=\{\mu\in P^{+}\mid(\mu_{\pi},\varepsilon_{i})\in\mathbb{Z},\ (\mu_{\pi},\delta_{j})\in\mathbb{Z}\}.
\]

\begin{rem}
\label{all integers}When studying the characters of simple finite
dimensional atypical modules, we may restrict without loss of generality
to the case that $\lambda\in\mathbb{P}^{+}$. See Remark 8 in \cite{CHR}. 
\end{rem}

\subsection{Weight diagrams and cap diagrams}

Diagrams encoding the weights of a module (among other things) were
introduced by Brundan and Stroppel in \cite{BS1} and were shown to
have numerous applications to the representation theory of $\mathfrak{gl}\left(m,n\right)$.
In particular, Brundan and Stroppel show that in some cases the corresponding
Khovanov algebra gives rise to a category equivalence with representations
of $GL(m,n)$ \cite{BS4}. Similar diagrams for $\mathfrak{osp}\left(m,2n\right)$
were used by Grusson and Serganova in \cite{GS} to give algorithmic
character formulas for basic classical Lie superalgebras. In this
paper, we restrict our attention to weight diagrams and cap diagrams
\cite{BS1}\@.

Let $\lambda\in\mathbb{P}^{+}$ and write 
\begin{equation}
\lambda^{\rho}=\sum_{i=1}^{m}a_{i}\varepsilon_{i}-\sum_{j=1}^{n}b_{j}\delta_{j}.\label{eq:lambda rho expansion}
\end{equation}
On the $\mathbb{Z}$-lattice, put $\vee$ above $t$ if $t\in\left\{ a_{i}\right\} \cap\left\{ b_{j}\right\} $,
put $\times$ above $t$ if $t\in\left\{ a_{i}\right\} \backslash\left\{ b_{j}\right\} $,
and put $\circ$ above $t$ if $t\in\left\{ b_{i}\right\} \backslash\left\{ a_{j}\right\} $.
If $t\not\in\{a_{i}\}\cup\{b_{j}\}$, then put $\wedge$. We refer
to such a diagram as a \emph{weight diagram}.

Note that each $\vee$ corresponds to some atypical root $\beta_{i}$.
We number the $\vee$'s left to right, and this is consistent our
chosen ordering for the set $S_{\lambda}$.

We obtain a \emph{cap diagram} from a weight diagram as follows. Going
from right to left and starting with the rightmost $\vee$, we draw
a cap connecting $\vee$ to first $\wedge$ to its right which is
\emph{``unmarked}'', and we say that this $\wedge$ is \emph{marked}
by $\vee$. By construction, the caps in our diagrams do not intersect. 
\begin{example}
\label{Ex: weight diagram pdc}If $\lambda^{\rho}=10\varepsilon_{1}+9\varepsilon_{2}+8\varepsilon_{3}+5\varepsilon_{4}+4\varepsilon_{5}-\delta_{1}-4\delta_{2}-6\delta_{3}-8\delta_{4}-10\delta_{5}$,
then the corresponding cap diagram $D_{\lambda}$ is %\begin{equation}%\stackrel{\wedge}{-1}\quad\stackrel{\wedge}{0}\quad\stackrel{\circ}{1}\quad\stackrel{\wedge}{2}\quad\stackrel{\wedge}{3}\quad\stackrel{\vee}{4}\quad\stackrel{\times}{5}\quad\stackrel{\circ}{6}\quad\stackrel{\wedge}{7}\quad\stackrel{\vee}{8}\quad\stackrel{\times}{9}\quad\stackrel{\vee}{10}\quad\stackrel{\wedge}{11}\quad\stackrel{\wedge}{12}.\label{eq:totally connected weight diagram}%\end{equation}

\begin{equation}
\text{\cd{14}{1/\we/-1, 2/\we/0, 3/\cir/1, 4/\we/2, 5/\we/3, 6/\ve/4, 7/\x/5, 8/\cir/6, 9/\we/7, 10/\ve/8, 11/\x/9, 12/\ve/10, 13/\we/11, 14/\we/12}{12/13/.33/.5, 10/14/.5/.7, 6/9/.5/.7}}
\end{equation}

\end{example}

\subsection{Characters and category $\mathcal{O}$}

Let $M$ be a module from the BGG category $\mathcal{O}$ \cite[8.2.3]{M}.
Then $M$ has a weight space decomposition $M=\oplus_{\mu\in\mathfrak{h}^{*}}M_{\mu}$,
where $M_{\mu}=\{x\in M\mid h.x=\mu(h)x\text{ for all }h\in\mathfrak{h}^{*}\}$,
and the \emph{character of $M$} is by definition $\mathrm{ch}\ M=\sum_{\mu\in\mathfrak{\mathfrak{h}^{*}}}\dim M_{\mu}\ e^{\mu}.$

Denote by $\mathcal{\mathcal{E}}$ the algebra of rational functions
$\mathbb{Q}(e^{\nu},\ \nu\in\mathfrak{h}^{*})$. The group $W$ acts
on $\mathcal{E}$ by mapping $e^{\nu}$ to $e^{w(\nu)}$. \textcolor{black}{For
$\beta\in\Delta_{\bar{1}}^{+}$, we identify elements of the form
$\frac{1}{1+e^{-\beta}}$ with their expansion as geometric series
in the domain $\left|e^{-\beta}\right|<1$. }Since $\Delta_{\bar{1}}^{+}$
is fixed by $W$, expanding commutes with the action of $W$.

The \emph{Weyl denominator of $\mathfrak{g}$} is defined to be 
\[
R=\frac{{\Pi_{\alpha\in\Delta_{\bar{0}}^{+}}(1-e^{-\alpha})}}{{\Pi_{\alpha\in\Delta_{\bar{1}}^{+}}(1+e^{-\alpha})}}.
\]
Then $e^{\rho}R$ is $W$-skew-invariant, i.e. $w(e^{\rho}R)=(-1)^{l(w)}e^{\rho}R$,
and $\mathrm{ch}\ L(\lambda)$ is $W$-invariant for $\lambda\in P^{+}$.
The character of a Verma module $M(\lambda)$ with $\lambda\in\mathfrak{h}^{*}$
is $\mathrm{ch}\ M(\lambda)=e^{\lambda}R^{-1}$. The character of
the Kac module $\overline{L}(\lambda)$ with $\lambda\in P^{+}$ is
\begin{equation}
\mathrm{ch\ }\overline{L}(\lambda)=\frac{1}{e^{\rho}R}\sum_{w\in W}(-1)^{l(w)}w(e^{\lambda^{\rho}}).\label{eq:Kac character}
\end{equation}

For $X\in\mathcal{E}$, we define 
\[
\mathcal{F}_{W}(X):=\sum_{w\in W}(-1)^{l(w)}w(X).
\]
We shall use the following lemma (see for example \cite[4.1.1]{G}). 
\begin{lem}
\label{lem:not regular}If $\nu\in\mathfrak{h}_{\mathbb{R}}^{*}$
is not regular, then $\mathcal{F}_{W}\left(e^{\nu}\right)=0$. 
\end{lem}

\subsection{Kazhdan-Lusztig polynomials and character formulas }

Generalized Kazhdan-Lusztig polynomials $K_{\lambda,\mu}\left(q\right)$
were introduced in \cite{S1} by Serganova to give an algorithmic
character formula for finite-dimensional irreducible representations
of $\mathfrak{gl}\left(m|n\right)$. Brundan gave a new algorithm
in \cite{B} for computing the generalized Kazhdan-Lusztig polynomials
for $\mathfrak{gl}(m|n)$ which can be described in terms of paths. 

We begin by recalling Brundan's algorithm \cite{B} for computing
$K_{\lambda,\mu}\left(q\right)$ using weight diagrams. We define
a\emph{ right move} map from the set of (labeled) weight diagrams
to itself in two steps. 
\begin{defn}
Let $D_{\mu}$ be a weight diagram for $\mu\in\mathbb{P}^{+}$. The
right move $R_{i}$ is defined by exchanging $\vee_{i}$ with the
$\wedge$ that it marks, that is, we switch $\vee_{i}$ with the $\wedge$
to which it is connected to by a cap in the corresponding cap diagram
of $\mu$.\end{defn}
\begin{example}
Let $D_{\mu}$ be the weight diagram for 
\[
\mu^{\rho}=8\varepsilon_{1}+7\varepsilon_{2}+6\varepsilon_{3}+2\varepsilon_{4}+1\varepsilon_{5}-2\delta_{1}-5\delta_{2}-7\delta_{3}-8\delta_{4}-9\delta_{5},
\]
 and consider the corresponding cap diagram:

\begin{center}
\cd{14}{1/\we/-1, 2/\we/0, 3/\cir/1, 4/\raisebox{10pt}{\hspace{4.5pt}$\vee^{1}$}/2,
5/\we/3, 6/\we/4, 7/\x/5, 8/\cir/6, 9/\raisebox{10pt}{\hspace{4.5pt}$\vee^{2}$}/7,
10/\raisebox{10pt}{\hspace{4.5pt}$\vee^{3}$}/8, 11/\x/9, 12/\we/10,
13/\we/11, 14/\we/12}{10/12/.33/.5, 9/13/.5/.7, 4/5/.33/.5} 
\par\end{center}

In this case, we see that $\vee_{1}$ marks 3, $\vee_{2}$ marks 11,
and $\vee_{3}$ marks 10. Hence

%\begin{eqnarray*}%R_{3}\left(D_{\mu}\right) & = & \stackrel{\wedge}{-1}\quad\stackrel{\wedge}{0}\quad\stackrel{\circ}{1}\quad\stackrel{\vee_{1}}{2}\quad\stackrel{\wedge}{3}\quad\stackrel{\wedge}{4}\quad\stackrel{\times}{5}\quad\stackrel{\circ}{6}\quad\stackrel{\vee_{2}}{7}\quad\stackrel{\wedge}{8}\quad\stackrel{\times}{9}\quad\stackrel{\vee_{3}}{10}\quad\stackrel{\wedge}{11}\quad\stackrel{\wedge}{12}\\%R_{2}\left(D_{\mu}\right) & = & \stackrel{\wedge}{-1}\quad\stackrel{\wedge}{0}\quad\stackrel{\circ}{1}\quad\stackrel{\vee_{1}}{2}\quad\stackrel{\wedge}{3}\quad\stackrel{\wedge}{4}\quad\stackrel{\times}{5}\quad\stackrel{\circ}{6}\quad\stackrel{\wedge}{7}\quad\stackrel{\vee_{3}}{8}\quad\stackrel{\times}{9}\quad\stackrel{\wedge}{10}\quad\stackrel{\vee_{2}}{11}\quad\stackrel{\wedge}{12}\\%R_{1}\left(D_{\mu}\right) & = & \stackrel{\wedge}{-1}\quad\stackrel{\wedge}{0}\quad\stackrel{\circ}{1}\quad\stackrel{\wedge}{2}\quad\stackrel{\vee_{1}}{3}\quad\stackrel{\wedge}{4}\quad\stackrel{\times}{5}\quad\stackrel{\circ}{6}\quad\stackrel{\vee_{2}}{7}\quad\stackrel{\vee_{3}}{8}\quad\stackrel{\times}{9}\quad\stackrel{\wedge}{10}\quad\stackrel{\wedge}{11}\quad\stackrel{\wedge}{12}.%\end{eqnarray*}
\begin{eqnarray*}
R_{3}\left(D_{\mu}\right) & = & \text{\cd{14}{1/\we/-1, 2/\we/0, 3/\cir/1, 4/\raisebox{10pt}{\hspace{4.5pt}\ensuremath{\vee^{1}}}/2, 5/\we/3, 6/\we/4, 7/\x/5, 8/\cir/6, 9/\raisebox{10pt}{\hspace{4.5pt}\ensuremath{\vee^{2}}}/7, 10/\we/8, 11/\x/9, 12/\raisebox{10pt}{\hspace{4.5pt}\ensuremath{\vee^{3}}}/10, 13/\we/11, 14/\we/12}{}}\\
R_{2}\left(D_{\mu}\right) & = & \text{\cd{14}{1/\we/-1, 2/\we/0, 3/\cir/1, 4/\raisebox{10pt}{\hspace{4.5pt}\ensuremath{\vee^{1}}}/2, 5/\we/3, 6/\we/4, 7/\x/5, 8/\cir/6, 9/\we/7, 10/\raisebox{10pt}{\hspace{4.5pt}\ensuremath{\vee^{3}}}/8, 11/\x/9, 12/\we/10, 13/\raisebox{10pt}{\hspace{4.5pt}\ensuremath{\vee^{2}}}/11, 14/\we/12}{}}\\
R_{1}\left(D_{\mu}\right) & = & \text{\cd{14}{1/\we/-1, 2/\we/0, 3/\cir/1, 4/\we/2, 5/\raisebox{10pt}{\hspace{4.5pt}\ensuremath{\vee^{1}}}/3, 6/\we/4, 7/\x/5, 8/\cir/6, 9/\raisebox{10pt}{\hspace{4.5pt}\ensuremath{\vee^{2}}}/7, 10/\raisebox{10pt}{\hspace{4.5pt}\ensuremath{\vee^{3}}}/8, 11/\x/9, 12/\we/10, 13/\we/11, 14/\we/12}{}}.
\end{eqnarray*}
\end{example}
\begin{defn}
\label{definition of path}Let $\lambda,\mu\in\mathbb{P}^{+}$. Label
the $\vee$'s in the diagram $D_{\mu}$ from left to right with $1,\ldots,r$.
A \emph{right path} from $D_{\mu}$ to $D_{\lambda}$ is a sequence
of right moves $\theta=R_{i_{1}}\circ\dots\circ R_{i_{k}}$ where
$i_{1}\le\ldots\le i_{k}$ and $\theta(D_{\mu})=D_{\lambda}$. The
length of the path is $l\left(\theta\right):=k$. \end{defn}
\begin{example}
Let $D_{\mu}$ be as in the previous example. Then $R_{1}\circ R_{1}\circ R_{2}\circ R_{3}\left(D_{\mu}\right)$
is the diagram $D_{\lambda}$ of Example \ref{Ex: weight diagram pdc}.
Note that after $R_{3}$ was applied to $D_{\mu}$, $\vee_{2}$ marks
8 since the spot became empty. 
\end{example}
Define a partial order on $P$ by $\mu^{\rho}\preceq\lambda^{\rho}$
if and only if $\lambda^{\rho}$ and $\mu^{\rho}$ have the same typical
entries, $\mathrm{atp}(\lambda^{\rho})=\mathrm{\mathrm{atp}(\mu^{\rho})}$
and the $i$-th atypical entry of $\mu^{\rho}$ is less than or equal
to the $i$-th atypical entry of $\lambda^{\rho}$ 
\begin{rem}
\label{remark on C Lexi}For each $\mu,\lambda\in\mathbb{P}^{+}$,
there exists a path from $D_{\mu}$ to $D_{\lambda}$ if and only
if $\mu^{\rho}\preceq\lambda^{\rho}$ \cite{B}. 
\end{rem}
Let $P_{\lambda,\mu}$ denote the set of paths from $D_{\mu}$ to
$D_{\lambda}$. If $P_{\lambda,\mu}$ is non-empty, it contains a
unique longest path, which sends the $i$-th $\vee$ of $\mu^{\rho}$
to the location of the $i$-th $\vee$ of $\lambda^{\rho}$. We call
this path the \emph{trivial path }from $D_{\mu}$ to $D_{\lambda}$
and denote its length by $l_{\lambda,\mu}$. 
\begin{lem}[{{{Brundan, \cite[Lemma 3.42]{B}}}}]
\label{thm:Brundan mod 2}For all $\lambda,\mu\in\mathbb{P}^{+}$
and $\theta\in P_{\lambda,\mu}$, $l(\theta)\text{\ensuremath{\equiv}}l_{\lambda,\mu}\ (\text{mod 2)}$. 
\end{lem}
We are now ready to state the result of Brundan and Serganova.
\begin{thm}[Serganova \cite{S1}, Brundan \cite{B}]
\label{char formula with KL polys}For each $\lambda\in\mathbb{P}^{+}$,
\[
\mbox{ch }L\left(\lambda\right)=\sum_{\mu\in\mathbb{P}^{+}}K_{\lambda,\mu}\left(-1\right)\mbox{ch }\overline{L}\left(\mu\right).
\]
where 
\[
K_{\lambda,\mu}\left(q\right)=\sum_{\theta\in P_{\lambda,\mu}}q^{l\left(\theta\right)}
\]
and $P_{\lambda,\mu}$ is the set of paths from $D_{\mu}$ to $D_{\lambda}$
and $l(\theta)$ denotes the length of the path $\theta$. 
\end{thm}
The following is a corollary of Theorem \ref{char formula with KL polys},
Lemma \ref{thm:Brundan mod 2} and Equation (\ref{eq:Kac character}). 
\begin{cor}
\label{cor:character with paths}Let $\lambda\in\mathbb{P}^{+}$,
and let $P_{\lambda}=\left\{ \mu\in\mathbb{P}^{+}\mid P_{\lambda,\mu}\text{ is non-empty}\right\} $.
Then 
\begin{equation}
e^{\rho}R\cdot\mbox{ch }L\left(\lambda\right)=\sum_{\mu\in P_{\lambda}}d_{\lambda,\mu}\cdot\left(-1\right)^{l_{\lambda,\mu}}\mathcal{F}_{W}\left(e^{\mu^{\rho}}\right)\label{eq:char formula with paths}
\end{equation}
where $d_{\lambda,\mu}$ is the number of paths from $D_{\mu}$ to
$D_{\lambda}$. 
\end{cor}

\section{Piecewise disconnected weights}

\subsection{\label{sub:Piecewise-disconnected-weights}Piecewise disconnected
weights}

We will see that some simple highest weight modules have particularly
nice character formulas. In this section we characterize their highest
weights.

The following definition is equivalent to that of \cite[Section 3.7]{SZ}. 
\begin{defn}
\label{def-tc}A weight $\lambda\in\mathbb{P}^{+}$ is called \emph{totally
connected} if in the weight diagram $D_{\lambda}$ between any two
$\vee$'s there is no $\wedge$. \label{def-tdc}A weight $\lambda\in\mathbb{P}^{+}$
is called\emph{ totally disconnected} if the diagram $D_{\lambda}$
contains at least one $\wedge$ between any two $\vee$'s.\end{defn}
\begin{rem}
\textcolor{black}{The cap diagram for a totally connected weight looks
like a rainbow. In particular, a weight is totally connected if and
only if its cap diagram satisfies the property that if a cap $A$
is below cap $B$, then all the caps that are below $B$ are either
above or below $A$. Whereas, a weight is totally disconnected if
and only if its cap diagram satisfies the property that no cap is
below another cap.} 
\end{rem}

\begin{rem}
A weight $\lambda\in\mathbb{P}^{+}$ is totally connected if and only
if for every $\mu\in\mathbb{P}^{+}$ the only possible path from $D_{\mu}$
to $D_{\lambda}$ is the trivial path, whereas it is totally disconnected
if and only if there exists $\mu\in\mathbb{P}^{+}$ with $r!$ paths
from $D_{\mu}$ to $D_{\lambda}$, where $r=\mathrm{atp}(\lambda^{\rho})$. \end{rem}
\begin{defn}
\label{definition of T}Let $\lambda\in\mathbb{P}^{+}$. We call a
nonempty continuous subsection of the weight diagram $D_{\lambda}$
an \emph{atypical component }if it contains an $\vee$, does not contain
any $\wedge$'s and is maximal with this property. If $\vee_{j}$
and $\vee_{k}$ belong to the same atypical component then we\emph{
write} $j\sim k$. \label{def-t_lambda}Enumerate the atypical components
of $D_{\lambda}$ left to right $T_{1},\ldots,T_{N}$, and let $t_{i}$
be the number of $\vee$'s contained in $T_{i}$ for $i=1,\ldots,N$.
We define $t_{\lambda}=t_{1}!t_{2}!\cdots t_{N}!$. 
\end{defn}

\begin{defn}
\label{def-pdc}We call a weight $\lambda\in\mathbb{P}^{+}$ and the
corresponding weight diagram $D_{\lambda}$ \emph{piecewise disconnected
(or PDC)} if $t_{i}\leq s_{i}$, where $s_{i}$ is the number of $\wedge$'s
between $T_{i}$ and $T_{i+1}$, for $i=1,\ldots,N-1$. \end{defn}
\begin{rem}
A weight is piecewise disconnected if and only if its cap diagram
satisfies the property that whenever two caps $A$ and $B$ are both
below the same cap $C,$ then either $A$ is below $B$, or $B$ is
below $A$. 
\end{rem}

\begin{rem}
\textcolor{black}{In the language of \cite[Section 14]{HW}, a weight
is piecewise disconnected if and only if the forest of $\lambda$
is a disjoint union of lines. In this case, $t_{\lambda}$ is equal
to the forest factorial. } \end{rem}
\begin{example}
The weight diagram $D_{\lambda}$ in Example \ref{Ex: weight diagram pdc}
is piecewise disconnected, but is neither totally connected nor totally
disconnected. It has two atypical components, namely, $T_{1}=\{4,5,6\},\ T_{2}=\{8,9,10\}$,
and $t_{1}=1$, $t_{2}=2$, $s_{1}=1$. 
\end{example}
The following lemma is a corollary of the definition. 
\begin{lem}
Any weight of atypicality $1$ or $2$ is either totally connected
or totally disconnected, and hence is piecewise disconnected. \end{lem}
\begin{example}
\label{ex:tc standard}The highest weight $\lambda=\varepsilon_{1}$
of the standard module $V$ of $\mathfrak{gl}(m|n)$ is totally connected,
as is the highest weight $\lambda=-\delta_{n}$ of the dual module
$V^{*}$. For example, for $\mathfrak{gl}(3|3)$ we have that $\lambda^{\rho}=4\varepsilon_{1}+2\varepsilon_{2}+1\varepsilon_{3}-1\delta_{1}-2\delta_{2}-3\delta_{3}$
and the cap diagram $D_{\lambda}$ is

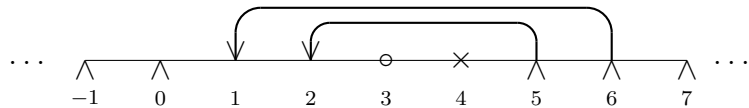
\[
\ldots\raisebox{-.7cm}{\text{\cd{9}{1/\we/-1, 2/\we/0, 3/\ve/1, 4/\ve/2, 5/\cir/3, 6/\x/4, 7/\we/5, 8/\we/6, 9/\we/7}{4/7/.25/.5,3/8/.33/.7}}}\ldots
\]

\end{example}

\begin{example}
The highest weight $\lambda$ of the non-trivial subquotient of the
adjoint module of $\mathfrak{gl}(2|2)$ is totally disconnected. Indeed,
$\lambda^{\rho}=3\varepsilon_{1}+1\varepsilon_{2}-1\delta_{1}-3\delta_{2}$
and the cap diagram $D_{\lambda}$ is

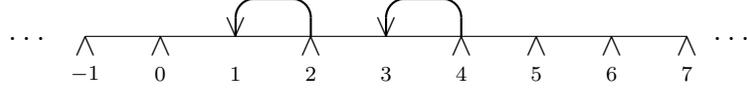
\[
\ldots\raisebox{-.7cm}{\text{\cd{9}{1/\we/-1, 2/\we/0, 3/\ve/1, 4/\we/2, 5/\ve/3, 6/\we/4, 7/\we/5, 8/\we/6, 9/\we/7}{5/6/.25/.5, 3/4/.25/.5}}}\ldots
\]

\end{example}

\begin{example}
\label{ex:non-pdc adjoint}For $n\geq3$, the highest weight $\lambda=\varepsilon_{1}-\delta_{n}$
of the non-trivial subquotient of the adjoint module of $\mathfrak{gl}(n|n)$
is not piecewise disconnected. For example, for $\mathfrak{gl}(3|3)$
we have that $\lambda^{\rho}=4\varepsilon_{1}+2\varepsilon_{2}+1\varepsilon_{3}-1\delta_{1}-2\delta_{2}-4\delta_{3}$
and the corresponding cap diagram $D_{\lambda}$ is 
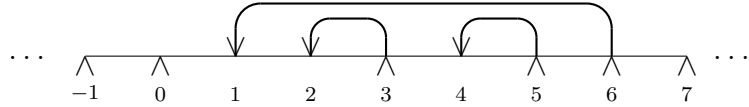
\[
\ldots\raisebox{-.7cm}{\text{\cd{9}{1/\we/-1, 2/\we/0, 3/\ve/1, 4/\ve/2, 5/\we/3, 6/\ve/4, 7/\we/5, 8/\we/6, 9/\we/7}{6/7/.25/.5,4/5/.25/.5,3/8/.33/.7}}}\ldots
\]

If $\mathfrak{g}=\mathfrak{gl}(n|n)$, then the adjoint module of
$\mathfrak{g}$ has a unique non-trivial simple subquotient $L(\nu)$.
The highest weight $\nu=\varepsilon_{1}-\delta_{n}$ is not piecewise
disconnected. 
\end{example}

\begin{example}
\label{ex:pdc not tc not tdc}If $\mathfrak{g}=\mathfrak{gl}(n|n)$
then the module $\mathfrak{g}\otimes\mathfrak{g\otimes\cdots\otimes g}$
obtained by tensoring $n-1$ copies of the adjoint module has a maximal
weight $\mu=(n-1)\nu=(n-1)\varepsilon_{1}-(n-1)\delta_{n}$ and a
simple subquotient $L(\mu)$. The weight $\nu$ is piecewise disconnected
but is neither totally connected nor totally disconnected. In particular
if $\mathfrak{g}=\mathfrak{gl}(3|3)$, then $L(\nu)$ with $\nu=2\varepsilon_{1}-2\delta_{n}$
is a simple subquotient of the module $\mathfrak{g}\otimes\mathfrak{g}$.
So $\nu^{\rho}=5\varepsilon_{1}+2\varepsilon_{2}+1\varepsilon_{3}-1\delta_{1}-2\delta_{2}-5\delta_{3}$
and the cap diagram for $D_{\nu}$ is

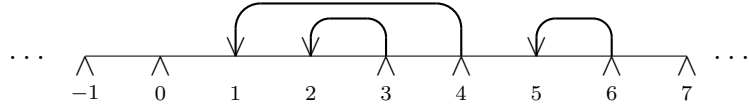
\[
\ldots\raisebox{-.7cm}{\text{\cd{9}{1/\we/-1, 2/\we/0, 3/\ve/1, 4/\ve/2, 5/\we/3, 6/\we/4, 7/\ve/5, 8/\we/6, 9/\we/7}{4/5/.25/.5,3/6/.33/.7,7/8/.25/.5}}}\ldots
\]

\end{example}

\subsection{Definition of $(\lambda^{\rho})^{\Uparrow}$ }

\textcolor{black}{The integral weight $(\lambda^{\rho})^{\Uparrow}$
is a modification of $\lambda^{\rho}$ which shall replace $\lambda^{\rho}$
in the character formula (Theorem \ref{thm:main theorem}). L}et $\lambda\in\mathbb{P}^{+}$
and write $\lambda^{\rho}$ as in (\ref{eq:lambda rho expansion}).
We refer to the coefficient $a_{i}$ (resp. $b_{j}$) as the\textit{
$\varepsilon_{i}$-entry} (resp. $\delta_{j}$-entry). If $\pm(\varepsilon_{k}-\delta_{l})\in S_{\lambda}$,
then we call the $\varepsilon_{k}$ and $\delta_{l}$ entries \emph{atypical}.
Otherwise, an entry is called \textit{typical}. 
\begin{defn}
\label{def: uparrow}If $\lambda\in\mathbb{P}^{+}$ is piecewise disconnected,
we denote by $(\lambda^{\rho})^{\Uparrow}$ the element obtained from
$\lambda^{\rho}$ by replacing each atypical entry with the maximal
atypical entry in the atypical component to which it belongs. \end{defn}
\begin{rem}
If $\lambda\in\mathbb{P}^{+}$ is totally disconnected then $(\lambda^{\rho})^{\Uparrow}=\lambda^{\rho}$,
whereas if $\lambda\in\mathbb{P}^{+}$ is totally connected then all
the atypical entries of $(\lambda^{\rho})^{\Uparrow}$ equal the maximal
atypical entry of $\lambda^{\rho}$.\end{rem}
\begin{example}
If $\lambda^{\rho}$ is as in Example \ref{Ex: weight diagram pdc},
then 
\begin{eqnarray*}
\left(\lambda^{\rho}\right)^{\Uparrow} & = & 10\varepsilon_{1}+9\varepsilon_{2}+10\varepsilon_{3}+5\varepsilon_{4}+4\varepsilon_{5}-\delta_{1}-4\delta_{2}-6\delta_{3}-10\delta_{4}-10\delta_{5}.
\end{eqnarray*}
\end{example}
\begin{defn}
\label{def: number of betas}If $\nu\in\mathfrak{h}^{*}$ can be written
as $\nu=\sum_{\alpha\in S_{\lambda}}k_{\alpha}\alpha$, then we define
\[
\left|\nu\right|_{S_{\lambda}}:=\sum_{\alpha\in S_{\lambda}}k_{\alpha}.
\]
Observe that $|(\lambda^{\rho})^{\Uparrow}-\lambda^{\rho}|_{S_{\lambda}}$
is a non-negative integer. 
\end{defn}

\section{main theorem}

The main theorem of this paper is as follows. 
\begin{thm}
\label{thm:main theorem}Let $\lambda\in\mathbb{P}^{+}$ be a piecewise
disconnected weight. Then 
\begin{equation}
e^{\rho}R\cdot\mbox{ch }L\left(\lambda\right)=\frac{(-1)^{|(\lambda^{\rho})^{\Uparrow}-\lambda^{\rho}|_{S_{\lambda}}}}{t_{\lambda}}\sum_{w\in W}(-1)^{l(w)}w\left(\frac{e^{(\lambda^{\rho})^{\Uparrow}}}{\prod_{\beta\in S_{\lambda}}\left(1+e^{-\beta}\right)}\right),\label{eq:main theorem}
\end{equation}
where $t_{\lambda}=t_{1}!t_{2}!\cdots t_{N}!$ (see Definition \ref{def-t_lambda})
and $S_{\lambda}$ is the (unique) $\lambda^{\rho}$-maximal isotropic
set of roots. \end{thm}
\begin{rem}
A totally connected weight $\lambda$ is piecewise disconnected with
$N=1$ and $t_{\lambda}=r!$. A totally disconnected weight $\lambda$
is piecewise disconnected with $N=r$ and $t_{\lambda}=1$. Here $r=\mathrm{atp}(\lambda^{\rho})$. 
\end{rem}

\subsection{\textcolor{black}{A map from the set of paths to $Sym(r)$}}

\textcolor{black}{One of the ideas of the proof is to translate the
character formula given in terms of paths in (\ref{eq:char formula with paths})
to a formula in terms of the Weyl group. For each $\lambda,\mu\in\mathbb{P}^{+}$,
we give an injective map from the set of paths $P_{\lambda,\mu}$
to $Sym(r)$, where $r$ is the atypicality of $\lambda$. We shall
later embed $Sym(r)$ in $W.$ We describe the image of this map when
$\lambda$ is piecewise disconnected. The image of such a map for
general $\lambda$ was described by Su and Zhang in \cite[Section 3.8]{SZ}.}

For $\lambda,\mu\in\mathbb{P}^{+}$, number the $\vee$'s of $D_{\mu}$
left to right $\vee_{1},\dots,\vee_{r}$ and number the $\check{\vee}$'s
of $D_{\lambda}$ left to right $\check{\vee}_{1},\ldots,\check{\vee}_{r}$.
Then a path $\theta\in P_{\lambda,\mu}$ determines uniquely an element
of $Sym(r)$ given by the ordering 
\[
\vee_{k}\mapsto\check{\vee}_{\sigma_{\theta}(k)}.
\]
In this way, we define the map $\Theta_{\lambda,\mu}:P_{\lambda,\mu}\rightarrow Sym(r)$.
The map $\Theta_{\lambda,\mu}$ is injective, since a path is determined
by this ordering. The image of the trivial path is the identity element
of $Sym(r)$. 
\begin{example}
\label{ex: non-pdc paths to permutations}Let $D_{\lambda}$ be as
in Example \ref{ex:non-pdc adjoint} and let $D_{\mu}$ be

\begin{equation}
\ldots\raisebox{-.7cm}{\text{\cd{9}{1/\we/-1, 2/\we/0, 3/\raisebox{10pt}{\hspace{4.5pt}\ensuremath{\vee^{1}}}/1, 4/\raisebox{10pt}{\hspace{4.5pt}\ensuremath{\vee^{2}}}/2, 5/\raisebox{10pt}{\hspace{4.5pt}\ensuremath{\vee^{3}}}/3, 6/\we/4, 7/\we/5, 8/\we/6, 9/\we/7}{}}}\ldots
\end{equation}

There are two paths from $D_{\mu}$ to $D_{\lambda}$, namely, the
trivial path and the path $R_{1}R_{1}R_{1}R_{2}R_{2}R_{2}R_{3}R_{3}$
which can be computed as follows. 
\begin{eqnarray*}
R_{3}R_{3}\left(D_{\mu}\right) & = & \ldots\raisebox{-.7cm}{\text{\cd{9}{1/\we/-1, 2/\we/0, 3/\raisebox{10pt}{\hspace{4.5pt}\ensuremath{\vee^{1}}}/1, 4/\raisebox{10pt}{\hspace{4.5pt}\ensuremath{\vee^{2}}}/2, 5/\we/3, 6/\we/4, 7/\raisebox{10pt}{\hspace{4.5pt}\ensuremath{\vee^{3}}}/5, 8/\we/6, 9/\we/7}{}}}\ldots\\
R_{2}R_{2}R_{2}R_{3}R_{3}\left(D_{\mu}\right) & = & \ldots\raisebox{-.7cm}{\text{\cd{9}{1/\we/-1, 2/\we/0, 3/\raisebox{10pt}{\hspace{4.5pt}\ensuremath{\vee^{1}}}/1, 4/\we/2, 5/\we/3, 6/\we/4, 7/\raisebox{10pt}{\hspace{4.5pt}\ensuremath{\vee^{3}}}/5, 8/\we/6, 9/\raisebox{10pt}{\hspace{4.5pt}\ensuremath{\vee^{2}}}/7}{}}}\ldots\\
D_{\lambda}=R_{1}R_{1}R_{1}R_{2}R_{2}R_{2}R_{3}R_{3}\left(D_{\mu}\right) & = & \ldots\raisebox{-.7cm}{\text{\cd{9}{1/\we/-1, 2/\we/0, 3/\we/1, 4/\we/2, 5/\we/3, 6/\raisebox{10pt}{\hspace{4.5pt}\ensuremath{\vee^{1}}}/4, 7/\raisebox{10pt}{\hspace{4.5pt}\ensuremath{\vee^{3}}}/5, 8/\we/6, 9/\raisebox{10pt}{\hspace{4.5pt}\ensuremath{\vee^{2}}}/7}{}}}\ldots.
\end{eqnarray*}
The image of this non-trivial path under the map $\Theta_{\lambda,\mu}$
is the cycle $\left(23\right)$. There are no other paths, because
if positions 4 and 5 were filled before position 7 then position 7
would be held, making the path impossible to complete. 
\end{example}
\textcolor{black}{For an element $\nu\in P$ with $\mbox{atp}\left(\nu\right)=r$
let $S_{\nu}=\left\{ \varepsilon_{m_{1}}-\delta_{n_{1}},\ldots,\varepsilon_{m_{r}}-\delta_{n_{r}}\right\} $
be such that $n_{1}<....<n_{r}$. We denote $\nu_{i}:=\left(\nu,\delta_{n_{i}}\right)$.
Then $\vee_{k}=\left(\mu^{\rho}\right)_{k}$ and that $\check{\vee}_{k}=\left(\lambda^{\rho}\right)_{k}$.}

In the following lemma we describe the image of $\Theta_{\lambda,\mu}$
for an arbitrary piecewise disconnected weight. 
\begin{lem}
\label{lem:pdc}If $\lambda\in\mathbb{P}^{+}$ is piecewise disconnected,
then 
\[
Im\ \Theta_{\lambda,\mu}=\left\{ \sigma\in Sym(r)\mid\sigma(\mu^{\rho})\preceq\lambda^{\rho},\text{ and }\sigma^{-1}(j)<\sigma^{-1}(k)\text{ if }j<k\text{ and }j\sim k\right\} ,
\]
where $j\sim k$ when $j$ and $k$ label $\check{\vee}$'s from the
same atypical component of $\lambda$.\end{lem}
\begin{proof}
Let $\theta\in P_{\lambda,\mu}$. Since the $\vee$'s move in order
from left to right to their respective destinations, we have that
$\vee_{k}\leq\check{\vee}_{\sigma_{\theta}(k)}$. This ensures that
$\sigma(\mu^{\rho})\preceq\lambda^{\rho}$. When an $\vee$ reaches
its destination, it marks the next $\wedge$ after it. Hence, the
$\vee$'s must go in order into each atypical component so that every
spot can be filled, that is, if $j<k$ and $j\sim k$ then $\sigma_{\theta}^{-1}(j)<\sigma_{\theta}^{-1}(k)$.
Hence, we always have inclusion. When $\lambda$ is piecewise disconnected,
these conditions on $\sigma\in Sym(r)$ are sufficient to define a
path $\theta$ from $D_{\mu}$ to $D_{\lambda}$ which satisfies $\vee_{k}\mapsto\check{\vee}_{\sigma_{\theta}(k)}$.
Indeed, the number of $\wedge$'s following an atypical component
and preceding the next is greater than or equal to the number of $\vee$'s
in a given atypical component, so an $\vee$ does not hold an $\check{\vee}$
spot.\end{proof}
\begin{rem}
If $\lambda$ is not piecewise disconnected then Lemma \ref{lem:pdc}
does not hold. See \linebreak{}
 \cite[Section 3.8]{SZ} for a description of the image in the general
case. 
\end{rem}
In the following lemma we change the defining conditions of the set
from Lemma \ref{lem:pdc} by replacing $\lambda^{\rho}$ with $\left(\lambda^{\rho}\right)^{\Uparrow}$,
and then we show that this does not change the set. 
\begin{lem}
\label{lem: image}If $\lambda\in\mathbb{P}^{+}$ is piecewise disconnected,
then

\begin{equation}
Im\ \Theta_{\lambda,\mu}=\left\{ \sigma\in Sym(r)\mid\sigma(\mu^{\rho})\preceq(\lambda^{\rho})^{\Uparrow},\text{ and }\sigma^{-1}(j)<\sigma^{-1}(k)\text{ if }j<k\text{ and }j\sim k\right\} .\label{eq:image of paths}
\end{equation}
\end{lem}
\begin{proof}
Let $A_{\lambda,\mu}=LHS$ and $B_{\lambda,\mu}=RHS$. By Lemma \ref{lem:pdc},
$A_{\lambda,\mu}\subseteq B_{\lambda,\mu}$. Now suppose towards a
contradiction that $\sigma\in B_{\lambda,\mu}\setminus A_{\lambda,\mu}$.
Choose $s$ maximal such that $\left(\lambda^{\rho}\right)_{\sigma(s)}<\left(\mu^{\rho}\right)_{s}\leq\left(\lambda^{\rho}\right)_{\sigma(s)}^{\Uparrow}$.
By definition $\left(\lambda^{\rho}\right)_{\sigma(s)}^{\Uparrow}=\left(\lambda^{\rho}\right)_{k}$,
where $k$ is the index of the maximal atypical entry in the atypical
component containing $\left(\lambda^{\rho}\right)_{\sigma(s)}$. Thus
$\left(\mu^{\rho}\right)_{s}=\left(\lambda^{\rho}\right)_{j}$ for
some $\sigma(s)<j\leqslant k$, since the atypical components of $\lambda^{\rho}$
are connected and $\mu^{\rho}$ is regular with the same typical entries
as $\lambda^{\rho}$. Thus $s<\sigma^{-1}(j)$ since $\sigma(s)\sim j$.
Then since $\mu^{\rho}$ is strictly dominant we have that $\left(\lambda^{\rho}\right)_{j}=\left(\mu^{\rho}\right)_{s}<\left(\mu^{\rho}\right)_{\sigma^{-1}(j)}$.
Note that we also have $\left(\mu^{\rho}\right)_{\sigma^{-1}(j)}\leq\left(\lambda^{\rho}\right)_{j}^{\Uparrow}$
since $\sigma\in B_{\lambda,\mu}$. This contradicts the maximality
of $s$, since $\sigma^{-1}(j)$ is larger and satisfies the required
properties. Hence $A_{\lambda,\mu}=B_{\lambda,\mu}$. 
\end{proof}

\subsection{\label{subsection: A bijection of sets}A bijection of indexing sets}

In this section, we change the indexing set of the character formula
in (\ref{eq:char formula with paths}) from $P_{\lambda}$ to a particular
subset of $\left(\lambda^{\rho}-\mathbb{N}S_{\lambda}\right)$.

Fix \emph{$\lambda\in\mathbb{P}^{+}$.} For each $\mu\in P_{\lambda}$,
the $W$ orbit of $\mu^{\rho}$ intersects $\left(\lambda^{\rho}-\mathbb{N}S_{\lambda}\right)$.
We denote by $\overline{\mu}$ the unique maximal element of this
intersection with respect to the standard order on $\mathfrak{h}^{*}$.
We define 
\[
C_{\lambda,\mathrm{reg}}^{\mathrm{Lexi}}:=\left\{ \overline{\mu}\in\lambda^{\rho}-\mathbb{N}S_{\lambda}\ \mid\ \mu\in P_{\lambda}\right\} .
\]
Since $P_{\lambda}\subset\mathbb{P}^{+}$, this defines a bijection
between the sets $P_{\lambda}$ and $C_{\lambda,\mathrm{reg}}^{\mathrm{Lexi}}$.
Recall that $S_{\lambda}=\left\{ \beta_{1},\ldots,\beta_{r}\right\} $
is ordered so that $\beta_{i}=\varepsilon_{p_{i}}-\delta_{q_{i}}$
and $q_{i}<q_{i+1}$. For $\nu\in(\lambda^{\rho})^{\Uparrow}-\mathbb{N}S_{\lambda}$
and $i=1,\ldots,r$, define 
\[
\nu_{\beta_{i}}=(\nu,\delta_{q_{i}}).
\]

\begin{lem}
One has 
\begin{eqnarray*}
C_{\lambda,\mathrm{reg}}^{\mathrm{Lexi}} & = & \left\{ \nu\in\lambda^{\rho}-\mathbb{N}S_{\lambda}\ \mid\ \nu_{\beta_{1}}<\nu_{\beta_{2}}<\ldots<\nu_{\beta_{r}}\ \mathrm{and\ \nu}\ \mathrm{is}\ \mathrm{regular}\right\} .
\end{eqnarray*}
\end{lem}
\begin{proof}
Clearly we have $\subseteq$, since $\mu^{\rho}$ is strictly dominant.
The reverse inclusion follows from Remark \ref{remark on C Lexi}
since for regular $\nu\in\lambda^{\rho}-\mathbb{N}S_{\lambda}$ and
$w\in W$ with $w(\nu)$ strictly dominant, $w(\nu)\preceq\lambda^{\rho}$
by definition.\end{proof}
\begin{defn}
\label{def of d bar}For $\bar{\mu}\in C_{\lambda,\mathrm{reg}}^{\mathrm{Lexi}}$,
define $\bar{d}{}_{\lambda,\bar{\mu}}$ to be the number of paths
from $D_{\mu}$ to $D_{\lambda}$, where $\mu$ is the unique dominant
element in the $W$ orbit of $\bar{\mu}$. 
\end{defn}
The following lemma is proven using techniques from \cite[Section 4.1]{SZ}. 
\begin{lem}
One has\label{lem:Lexi sum} 
\[
e^{\rho}R\cdot\mbox{ch }L\left(\lambda\right)=\sum_{\bar{\mu}\ \in\ C_{\lambda,\mathrm{reg}}^{\mathrm{Lexi}}}\bar{d}{}_{\lambda,\bar{\mu}}\left(-1\right)^{\left|\lambda^{\rho}-\bar{\mu}\right|_{S_{\lambda}}}\mathcal{F}_{W}\left(e^{\bar{\mu}}\right).
\]
\end{lem}
\begin{proof}
By Corollary \ref{cor:character with paths} it suffices to show that
for each $\mu\in P_{\lambda}$, 
\[
\left(-1\right)^{l_{\lambda,\mu}}\mathcal{F}_{W}\left(e^{\mu^{\rho}}\right)=\left(-1\right)^{\left|\lambda^{\rho}-\overline{\mu}\right|_{S_{\lambda}}}\mathcal{F}_{W}\left(e^{\overline{\mu}}\right).
\]
Let $w'\in W$ such that $w'(\mu^{\rho})=\overline{\mu}$. To complete
the proof it is sufficient to show that $\left|\lambda^{\rho}-\overline{\mu}\right|_{S_{\lambda}}=l_{\lambda,\mu}+l\left(w'\right)$.
The number $\left|\lambda^{\rho}-\overline{\mu}\right|_{S_{\lambda}}$
is the sum of the differences between the atypical entries of $\lambda^{\rho}$
and $\overline{\mu}$. This is equal to the number of moves in the
trivial path $l_{\lambda,\mu}$ plus the number of spots being skipped.
We will show that $l\left(w'\right)$ is exactly the number of spots
skipped in the trivial path.

The element $w'\in W$ for which $w'(\mu^{\rho})=\overline{\mu}$
can be described explicitly in terms of the trivial path $\theta$.
Denote $\theta=R_{i_{1}}\circ\dots\circ R_{i_{N}}$, then $w'=w_{1}\cdot\dots\cdot w_{N}$
where each $w_{j}$ is defined as follows. Suppose that the move $R_{i_{j}}$
moved the $\vee$ at $n_{j}$ to an $\wedge$ at $n_{j}+k_{j}+1$,
namely, it skipped over $k_{j}$ spots with $\times$'s and $\circ$'s.
Then $w_{j}=s_{1}\cdot\dots\cdot s_{k_{j}-1}$ where $s_{i}$ is of
the form $s_{\varepsilon_{l}-\varepsilon_{l+1}}$ if the $i$-th skip
is over the $\times$ of $\varepsilon_{l}$ and is of the form $s_{\delta_{l}-\delta_{l+1}}$
if it is over the $\circ$ of $\delta_{l}$. It is easy to see that
the expression is reduced, so $l\left(w_{j}\right)=k_{j}$ is the
number of spots skipped in the move $R_{i_{j}}$. Also $l\left(w'\right)=\sum l\left(w_{i}\right)$,
so $l\left(w'\right)$ is exactly the number of spots skipped in the
trivial path. 
\end{proof}

\subsection{Paths and permutations for piecewise disconnected weights}

In this section, we show that if $\lambda\in\mathbb{P}^{+}$ is a
piecewise disconnected weight, then for each $\mu\in P_{\lambda}$
there exists a $t_{\lambda}$ to $1$ map from the set of paths from
$\mu$ to $\lambda$ to a certain subset of the Weyl group. This is
a crucial step in the proof of the main theorem.

Let $W_{r}$ be the subgroup of $W$ that permutes $S_{\lambda}$.
Then $W_{r}\cong Sym(r)$ and is generated by elements of the form
$s{}_{\varepsilon_{i}-\varepsilon_{j}}s_{\delta_{i'}-\delta_{j'}}$
where $\varepsilon_{i}-\delta_{i'},\varepsilon_{j}-\delta_{j'}\in S_{\lambda}$.
So $\left|W_{r}\right|=r!$ and all $w\in W_{r}$ have positive sign.

Fix $\lambda\in\mathbb{P}^{+}$, and recall the notation of Section
\ref{sub:Piecewise-disconnected-weights}. We define a subgroup of
$W_{r}$ that preserves the atypical components of $\lambda^{\rho}$,
that is,

\begin{equation}
W_{r}(t_{\lambda})=\left\langle s{}_{\varepsilon_{i}-\varepsilon_{j}}s_{\delta_{i'}-\delta_{j'}}\mid i\sim j\right\rangle .
\end{equation}
So $w\in W_{r}(t_{\lambda})$ and $\lambda_{\beta}\in T_{i}$ imply
that $\lambda_{w(\beta)}\in T_{i}$. Clearly, 
\[
W_{r}(t_{\lambda})\cong Sym(t_{1})\vee\cdots\vee Sym(t_{N})
\]
and hence $W_{r}(t_{\lambda})$ has cardinality $t_{\lambda}$. 
\begin{defn}
\label{def of c}For each $\nu\in C_{\lambda,\mathrm{reg}}^{\mathrm{Lexi}}$,
let 
\[
W_{r}(\lambda,\nu):=\left\{ w\in W_{r}\mid w(\nu)\in(\lambda^{\rho})^{\Uparrow}-\mathbb{N}S_{\lambda}\right\} ,
\]
and let $c_{\lambda,\nu}=|W_{r}(\lambda,\nu)|.$ 
\end{defn}
Then 
\begin{equation}
\mathcal{F}_{W}\left(\sum_{w\in W_{r}(\lambda,\nu)}e^{w(\nu)}\right)=c_{\lambda,\nu}\cdot\mathcal{F}_{W}\left(e^{\nu}\right).\label{eq:rewrite sum}
\end{equation}

\begin{prop}
\label{prop:use pdc}Let $\lambda\in\mathbb{P}^{+}$ be a piecewise
disconnected weight. Then for every $\mu\in P_{\lambda}$, the number
of paths from $D_{\mu}$ to $D_{\lambda}$ equals $\frac{1}{t_{\lambda}}|W_{r}(\lambda,\bar{\mu})|$\textup{.
Hence, for each $\nu\in C_{\lambda,\mathrm{reg}}^{\mathrm{Lexi}}$,
we have that $\frac{\bar{d}{}_{\lambda,\nu}}{c_{\lambda,\nu}}=\frac{1}{t_{\lambda}}$.}\end{prop}
\begin{proof}
First, we observe that there is a natural bijection between the sets
$W_{r}(\lambda,\overline{\mu})$ and 
\[
\widetilde{B}_{\lambda,\mu}=\left\{ \sigma\in Sym(r)\mid\sigma(\mu^{\rho})\preceq(\lambda^{\rho})^{\Uparrow}\right\} ,
\]
since the bijective map $P_{\lambda}\rightarrow C_{\lambda,\mathrm{reg}}^{\mathrm{Lex}}$
defined by $\mu^{\rho}\mapsto\overline{\mu}$ preserves the relative
order of the atypical roots. So we may in fact identify $W_{r}(\lambda,\overline{\mu})$
with $\widetilde{B}_{\lambda,\mu}$ under this correspondence.

Now by Lemma \ref{lem: image}, $d_{\lambda,\mu}:=|P_{\lambda,\mu}|$
equals the cardinality of the set in (\ref{eq:image of paths}), which
we denote by $B_{\lambda,\mu}$. We claim that there is a bijection
of sets $W_{r}(t_{\lambda})\vee B_{\lambda,\mu}\cong\widetilde{B}_{\lambda,\mu}$
defined by $(w,\sigma)\mapsto w\sigma$. Now by definition, $\left(\lambda^{\rho}\right)_{j}^{\Uparrow}=\left(\lambda^{\rho}\right)_{k}^{\Uparrow}$
when $\lambda_{j}^{\rho}$ and $\lambda_{k}^{\rho}$ belong to the
same atypical component, that is, when $j\sim k$. Since $W_{r}(t_{\lambda})$
preserves each atypical component, the map is well-defined, that is,
$\sigma(\mu^{\rho})\preceq(\lambda^{\rho})^{\Uparrow}$ implies that
$w\sigma(\mu^{\rho})\preceq(\lambda^{\rho})^{\Uparrow}$ for any $w\in W_{r}(t_{\lambda})$.

If $\sigma\in B_{\lambda,\mu}$, then the atypical entries of each
atypical component of $\sigma(\mu^{\rho})$ are in increasing order
and distinct, since $\sigma\in B_{\lambda,\mu}$ satisfies: $\sigma^{-1}(j)<\sigma^{-1}(k)$
when $j<k$ and $j\sim k$. It is not difficult to show that the map
defined above is bijective. Indeed, given $\sigma'\in\widetilde{B}_{\lambda,\mu}$
there exists a unique $w\in W_{r}(t_{\lambda})$ such that the atypical
entries of each atypical component of $w^{-1}\sigma'(\mu^{\rho})$
are in increasing order, that is, such that $w^{-1}\sigma'\in B_{\lambda,\mu}$.
Therefore, $W_{r}(t_{\lambda})\vee B_{\lambda,\mu}\cong\widetilde{B}_{\lambda,\mu}$
and $t_{\lambda}\cdot d_{\lambda,\mu}=c_{\lambda,\overline{\mu}}$.\end{proof}
\begin{example}
If $\lambda\in\mathbb{P}^{+}$ is not piecewise disconnected, then
the ratio $\frac{\bar{d}{}_{\lambda,\nu}}{c_{\lambda,\nu}}$ is not
necessarily constant. Consider the weight $\lambda$ from Example
\ref{ex:non-pdc adjoint}. If $\mu$ is the weight from Example \ref{ex: non-pdc paths to permutations}
then $\bar{d}{}_{\lambda,\bar{\mu}}=2$ and $c_{\lambda,\bar{\mu}}=6$,
whereas, if $\mu=\lambda$ then $\bar{d}{}_{\lambda,\bar{\mu}}=1$
and $c_{\lambda,\bar{\mu}}=2$. 
\end{example}

\subsection{Enlarging the indexing set}

In this section, we enlarge the indexing set $C_{\lambda,\mathrm{reg}}^{\mathrm{Lexi}}$
by adding non-regular elements, namely, we define 
\[
\overline{C_{\lambda}^{\mathrm{Lexi}}}=\left\{ \nu\in(\lambda^{\rho})^{\Uparrow}-\mathbb{N}S_{\lambda}\ \mid\ \nu_{\beta_{1}}<\nu_{\beta_{2}}<\ldots<\nu_{\beta_{r}}\right\} .
\]

\begin{lem}
\label{lem:add bar}If $\nu\in\overline{C_{\lambda}^{\mathrm{Lexi}}}\setminus C_{\lambda,\mathrm{reg}}^{\mathrm{Lexi}}$,
then $\nu$ is not regular. \end{lem}
\begin{proof}
Let $j$ be such that $\lambda_{\beta_{j}}^{\rho}<\nu_{\beta_{j}}\leq(\lambda^{\rho})_{\beta_{j}}^{\Uparrow}$
and $\nu_{\beta_{i}}\leq\lambda_{\beta_{i}}^{\rho}$ for all $i>j$.
By definition of $(\lambda^{\rho})^{\Uparrow}$, all the integers
between $\lambda_{\beta_{j}}^{\rho}+1$ and $((\lambda^{\rho})^{\Uparrow})_{\beta_{j}}$
are entries of $\lambda^{\rho}$. The typical entries of $\nu$ are
the same as of $\lambda^{\rho}$ and there are $r-j+1$ atypical entries
which are strictly greater than $\lambda_{\beta_{j}}^{\rho}$. This
implies that there must be equal entries of the same type, and hence
$\nu$ is not regular. \end{proof}
\begin{lem}
\label{lem:action by W}Let $\mathfrak{C}_{\lambda}=\left\{ w(\nu)\in\left(\lambda^{\rho}\right)^{\Uparrow}-\mathbb{N}S_{\lambda}\mid w\in W_{r},\ \nu\in\overline{C_{\lambda}^{\mathrm{Lexi}}}\right\} $
and 
\[
\mathfrak{D}_{\lambda}=\left\{ \nu\in(\lambda^{\rho})^{\Uparrow}-\mathbb{N}S_{\lambda}\ \mid\ \nu_{\beta_{i}}\neq\nu_{\beta_{j}}\ \text{for any }i\neq j\right\} .
\]
Then $\mathfrak{C}_{\lambda}=\mathfrak{D}_{\lambda}$ as multisets,
and hence elements of $((\lambda^{\rho})^{\Uparrow}-\mathbb{N}S_{\lambda})\setminus\mathfrak{C}_{\lambda}$
are not regular.\end{lem}
\begin{proof}
Clearly we have $\mathfrak{C}_{\lambda}\subseteq\mathfrak{D}_{\lambda}$
as sets. Since there is a unique element in the $W_{r}$ orbit of
any $\nu\in\overline{C_{\lambda}^{\mathrm{Lexi}}}$ that satisfies
$\nu_{\beta_{1}}<\nu_{\beta_{2}}<\ldots<\nu_{\beta_{r}}$, the orbits
of distinct elements from $\overline{C_{\lambda}^{\mathrm{Lexi}}}$
do not intersect. Hence, we have an inclusion of multisets. For the
reverse inclusion, suppose that $\nu\in\mathfrak{D}_{\lambda}$. Take
$\sigma\in W_{r}$ such that $\sigma^{-1}(\nu)$ satisfies $\nu_{\beta_{\sigma(1)}}<\nu_{\beta_{\sigma(2)}}<\cdots<\nu_{\beta_{\sigma(r)}}$.
Since 
\[
\nu_{\beta_{\sigma(i)}}\leq\mathrm{max}\{\nu_{\beta_{1}},\ldots,\nu_{\beta_{i}}\}\leq(\lambda^{\rho})_{\beta_{i}}^{\Uparrow}
\]
we have that $\sigma^{-1}(\nu)\in\left(\lambda^{\rho}\right)^{\Uparrow}-\mathbb{N}S_{\lambda}$.
Hence $\sigma^{-1}(\nu)\in\overline{C_{\lambda}^{\mathrm{Lexi}}}$
and $\nu=\sigma\left(\sigma^{-1}(\nu)\right)\in\mathfrak{C}_{\lambda}$
. 
\end{proof}

\subsection{Proof of the main theorem}
\begin{proof}[Proof of Theorem \ref{thm:main theorem}]
By Lemma \ref{lem:Lexi sum}, we have that 
\[
e^{\rho}R\cdot\mbox{ch }L\left(\lambda\right)=\sum_{\nu\ \in\ C_{\lambda,\mathrm{reg}}^{\mathrm{Lexi}}}\bar{d}{}_{\lambda,\nu}\cdot\left(-1\right)^{\left|\lambda^{\rho}-\nu\right|_{S_{\lambda}}}\mathcal{F}_{W}\left(e^{\nu}\right)
\]
which by (\ref{eq:rewrite sum}) equals 
\[
(-1)^{|(\lambda^{\rho})^{\Uparrow}-(\lambda^{\rho})|_{S_{\lambda}}}\sum_{\nu\ \in\ C_{\lambda,\mathrm{reg}}^{\mathrm{Lexi}}}\frac{\bar{d}{}_{\lambda,\nu}}{c_{\lambda,\nu}}\left(-1\right)^{\left|(\lambda^{\rho})^{\Uparrow}-\nu\right|_{S_{\lambda}}}\mathcal{F}_{W}\left(\sum_{w\in W_{r}(\lambda,\nu)}e^{w(\nu)}\right).
\]
Then by Proposition \ref{prop:use pdc} the latter is equal to 
\[
(-1)^{|(\lambda^{\rho})^{\Uparrow}-(\lambda^{\rho})|_{S_{\lambda}}}\sum_{\nu\ \in\ C_{\lambda,\mathrm{reg}}^{\mathrm{Lexi}}}\frac{1}{t_{\lambda}}\left(-1\right)^{\left|(\lambda^{\rho})^{\Uparrow}-\nu\right|_{S_{\lambda}}}\mathcal{F}_{W}\left(\sum_{w\in W_{r}(\lambda,\nu)}e^{w(\nu)}\right)
\]
and so by Lemma \ref{lem:add bar} and Lemma \ref{lem:not regular}
we have that it is equal to

\begin{eqnarray*}
 &  & \frac{(-1)^{|(\lambda^{\rho})^{\Uparrow}-(\lambda^{\rho})|_{S_{\lambda}}}}{t_{\lambda}}\sum_{\nu\ \in\ \overline{C_{\lambda}^{\mathrm{Lexi}}}}\left(-1\right)^{\left|(\lambda^{\rho})^{\Uparrow}-\nu\right|_{S_{\lambda}}}\mathcal{F}_{W}\left(\sum_{w\in W_{r}(\lambda,\nu)}e^{w(\nu)}\right).
\end{eqnarray*}
Then Lemma \ref{lem:action by W} and Lemma \ref{lem:not regular},
it is equal to

\[
\frac{(-1)^{|(\lambda^{\rho})^{\Uparrow}-(\lambda^{\rho})|_{S_{\lambda}}}}{t_{\lambda}}\sum_{\nu\ \in\ (\lambda^{\rho})^{\Uparrow}-\mathbb{N}S_{\lambda}}\left(-1\right)^{\left|(\lambda^{\rho})^{\Uparrow}-\nu\right|_{S_{\lambda}}}\mathcal{F}_{W}\left(e^{\nu}\right)
\]
which can be rewritten as 
\[
\frac{(-1)^{|(\lambda^{\rho})^{\Uparrow}-(\lambda^{\rho})|_{S_{\lambda}}}}{t_{\lambda}}\mathcal{F}_{W}\left(\sum_{\nu\ \in\ (\lambda^{\rho})^{\Uparrow}-\mathbb{N}S_{\lambda}}\left(-1\right)^{\left|(\lambda^{\rho})^{\Uparrow}-\nu\right|_{S_{\lambda}}}e^{\nu}\right)
\]
\[
=\frac{(-1)^{|(\lambda^{\rho})^{\Uparrow}-(\lambda^{\rho})|_{S_{\lambda}}}}{t_{\lambda}}\mathcal{F}_{W}\left(\frac{e^{(\lambda^{\rho})^{\Uparrow}}}{\prod_{\beta\in S_{\lambda}}\left(1+e^{-\beta}\right)}\right).
\]
\end{proof}

M.C.: Dept. of Mathematics, University of Minnesota, mchmutov@umn.edu\\
 C.H.: Dept. of Mathematics, Weizmann Institute of Science, crystal.hoyt@weizmann.ac.il\\
 S.R.: Dept. of Mathematics, ORT Braude College, shifi@braude.ac.il 
\end{document}